\documentclass[reqno,12pt]{amsart} 
\usepackage{vmargin}
\setpapersize{A4}
\usepackage{amsmath} 
\usepackage{amsfonts}   
\usepackage{amssymb}      
\usepackage{eufrak}      
\usepackage{amscd}      
\usepackage{amsthm}      
\usepackage{epsfig}      
\usepackage{amstext}      
\usepackage[all,line,dvips]{xy} 
\newcommand{\id}{\operatorname{id}}

\newcommand{\Proj}{\operatorname{Proj}}
\renewcommand{\Im}{\operatorname{Im}}
\newcommand{\di}{\operatorname{dir}}

   \theoremstyle{plain}%default
   \newtheorem{thm}{Theorem}[section]
   \newtheorem{prop}[thm]{Proposition}
     
   \newtheorem{cor}[thm]{Corollary}
   \theoremstyle{definition}

   \newtheorem{example}[thm]{Example}
   \theoremstyle{remark}
   
   \newtheorem{remark}[thm]{Remark}

\usepackage{graphicx}
   \numberwithin{equation}{section}

\title{Extensions of Hilbert $C^*$-modules: classification in simple cases}
  
\author{Vladimir Manuilov and Zhu Jingming}

\email{manuilov@mech.math.msu.su}
\email{zhugjingming\underline{\phantom{a}}math@163.com}

\address{Harbin Institute of Technology\\ 92 West Dazhi Str.\\ Harbin, Heilongjiang \\ 150001,
P. R. China\\ and Dept. of Mech. and Math.\\
Moscow State University\\
Leninskie Gory, Moscow\\ 
119991, Russia}
\address{Harbin Institute of Technology\\ 92 West Dazhi Str.\\ Harbin, Heilongjiang \\ 150001,
P. R. China}

\thanks{The first named author acknowledges partial support of the RFBR grant No. 10-01-00257 and of the Russian Government grant No. 11.G34.31.0005.}

\begin{document}

\begin{abstract}
Theory of extensions of Hilbert $C^*$-modules was developed by D. Baki\'c and B. Gulja\u s. An easy observation shows that in the case, when the underlying $C^*$-algebra extension is commutative and the Hilbert $C^*$-modules are projective of finite type, the algebraic properties of the corresponding Busby invariant allow to identify extensions with isometric maps of the corresponding vector bundles. When the Hilbert $C^*$-modules are free of rank one, we evaluate the set of extensions in topological terms.    

\end{abstract}

\maketitle

\section{Introduction}

Extensions of Hilbert $C^*$-modules were introduced by D. Baki\'c and B. Gulja\u s in \cite{BG}. 
Given an essential extension 
\begin{equation}\label{-1}
\begin{xymatrix}{
0\ar[r]& A\ar[r]& B\ar[r]& C\ar[r]& 0
}\end{xymatrix}
\end{equation}
of $C^*$-algebras, a short exact sequence 
\begin{equation}\label{-2}
\begin{xymatrix}{
0\ar[r]& V\ar[r]& W\ar[r]& Z\ar[r]& 0,
}\end{xymatrix}
\end{equation}
where $V$, $W$ and $Z$ are Hilbert $C^*$-modules over $A$, $B$ and $C$, respectively, is a Hilbert $C^*$-module extension if the connecting maps and inner products in (\ref{-2}) are compatible with the connecting maps in (\ref{-1}). In some aspects, Hilbert $C^*$-module extensions are similar to $C^*$-algebra extensions. In particular, 
D. Baki\'c and B. Gulja\u s in \cite{BG2} have constructed the multiplier Hilbert $C^*$-module and the Busby invariant for Hilbert $C^*$-module extensions with properties similar to those for $C^*$-algebras. But classification of Hilbert $C^*$-module extensions is much more difficult than that of $C^*$-algebra extensions. One of the reasons for that is the abundance of Hilbert $C^*$-modules over a $C^*$-algebra. We derive here some simple consequences from their theory. In our attempt to classify Hilbert $C^*$-module extensions we restrict ourselves to projective modules of finite type. Since such modules are easier to handle in the commutative case, we also restrict ourselves to the case of commutative $C^*$-algebras. In the case, when the $C^*$-algebras are commutative and the Hilbert $C^*$-modules $V$ and $Z$ are projective of finite rank we show that the homotopy classes of extensions (\ref{-2}) can be described in terms of isometries of the corresponding vector bundles. If, moreover, $V$ and $Z$ are free of rank one we show that homotopy equivalence classes the extensions (\ref{-2}) are classified by the first \v Cech cohomology of the Stone-\v Cech corona of the Gelfand spectrum $U=\widehat{A}$ of the $C^*$-algebra $A$.  

The first named author is grateful to the Institut Henri Poincar\'e for hospitality.

\section{Basic definitions}

Let us recall the basic definitions for Hilbert $C^*$-module extensions from \cite{BG,BG2}.

Let $V$ and $W$ be (right) Hilbert $C^*$-modules over $C^*$-algebras $A$ and $B$ respectively. For a $C^*$-algebra morphism $\phi$: $A\rightarrow B$, a map $\Phi$: $V\rightarrow W$ is called a $\phi$-morphism of Hilbert $C^*$-modules (or just Hilbert $C^*$-module morphism when one doesn't need to specify $\phi$) if $\langle\Phi(v),\Phi(v')\rangle =\phi(\langle v,v'\rangle)$ holds for all $v,v'\in V$. It is known that a $\phi$-morphism is linear, contractive and respects the module structure. If $A$ is an ideal in $B$ then the ideal submodule $WA$ of $W$ is defined as 
$$
WA =\{wa : w\in W, a\in A\} = \{w\in W : \langle w,w\rangle\in A\}. 
$$
For any quotient map $\pi$: $B\rightarrow B/A$, we have a canonical quotient map $\Pi$: $W\rightarrow W/WA$, which is obviously a $\pi$-morphism. If $V$ is full and $\phi$ is injective then $\Phi$ is injective. 

If $\Im\Phi=WA$ then there is a well defined quotient Hilbert $C^*$-module $Z=W/\Im\Phi$ over $B/A$. In this case the middle term $W$ of the short exact sequence
\begin{equation}\label{2}
\begin{xymatrix}{
0\ar[r]& V\ar[r]^-{\Phi}& W\ar[r]^-{\Pi}& Z\ar[r]& 0
}\end{xymatrix}
\end{equation}
is an extension of $Z$ by $V$ with the underlying $C^*$-algebra extension
\begin{equation}\label{1}
\begin{xymatrix}{
0\ar[r]& A\ar[r]^-{\phi}& B\ar[r]^-{\pi}& B/A\ar[r]& 0
}\end{xymatrix}
\end{equation}
(here $\Phi$ is a $\phi$-morphism and $\Pi$ is a $\pi$-morphism). 

Similarly to $C^*$-algebra extensions, D. Baki\'c and B. Gulja\u s in \cite{BG} have defined, for a full Hilbert $C^*$-module $V$ over $A$, the multiplier Hilbert module $M(V)$ over the multiplier $C^*$-algebra $M(A)$ (as the module of all adjointable module homomorphisms from $A$ to $V$; note that they used the notation $V_d$ for $M(V)$), and the quotient Hilbert $C^*$-module $Q(V)=M(V)/V$ over $Q(A)=M(A)/A$. This gives the `maximal' extension 
$$
\begin{xymatrix}{
0\ar[r]& V\ar[r]^-{\Gamma}& M(V)\ar[r]^-{\Sigma}& Q(V)\ar[r]& 0,
}\end{xymatrix}
$$
in which $\Gamma(v)=l_{v}$ is defined by $l_v(a)=va$, $v\in V$, $a\in A$,
with the underlying $C^*$-algebra extension
$$
\begin{xymatrix}{
0\ar[r]& A\ar[r]^-{\gamma}& M(A)\ar[r]^-{\sigma}& Q(A)\ar[r]& 0
}\end{xymatrix}
$$

Note that a $C^*$-algebra extension (\ref{1}) determines $C^*$-algebra morphisms $\lambda$:$B\to M(A)$ and $\delta:B/A\to Q(A)$, the latter being the Busby invariant for the extension (\ref{1}).
Maximality means, by \cite{BG}, Thm 1.2, that for a Hilbert $C^*$-module extension (\ref{2}), there is a Hilbert module $\lambda$-morphism $\Lambda$:$W\to M(V)$ and a $\delta$-morphism $\Delta:Z\to Q(V)$ such that the diagram 
\begin{equation}\label{d}
\begin{xymatrix}{
0\ar[r]& V\ar[r]^-{\Phi}\ar@{=}[d]& W\ar[r]^-{\Pi}\ar[d]^-{\Lambda}& Z\ar[r]\ar[d]^-{\Delta}& 0\\
0\ar[r]& V\ar[r]^-{\Gamma}& M(V)\ar[r]^-{\Sigma}& Q(V)\ar[r]& 0
}\end{xymatrix}
\end{equation}
commutes.

In the diagram (\ref{d}), the map $\Delta$ is called the Busby invariant corresponding to the extension
(\ref{2}). It is shown in \cite{BG2} that for an arbitrary Hilbert $C^*$-module $V$ and a given morphism
$\Delta: Z\to Q(V)$, there is (unique up to a unitary equivalence) an extension
$W$ such that $\Delta$ is the Busby invariant for it.

%Note that even if $V$ is full then $M(V)$ and $Q(V)$ need not to be full (although they are full in most examples). If %they are not full then we can pass to $<M(V),M(V)>\subset M(A)$ and $<Q(V),Q(V)>\subset Q(A)$ respectively. Then they %become full Hilbert $C^*$-modules over these $C^*$-algebras.  

Prop. 3.4 of \cite{BG2} states that if $V$ and $Z$ are full Hilbert $C^*$-modules then for any $\delta$-morphism $\Delta:Z\to Q(V)$ (for some $\delta:B/A\to Q(A)$) there exists a Hilbert $C^*$-module $W$ over $B$ such that the Busby invariant $\Delta_W$ of the extension (\ref{2}) equals $\Delta$. Th. 3.6 of \cite{BG2} specifies that if $W$ is assumed to be full too then there is a one-to-one correspondence between unitary equivalence classes of extensions (\ref{2}) and Hilbert $C^*$-module morphisms $\Delta:Z\to Q(V)$.

Besides the unitary equivalence, it is natural to consider a more flexible equivalence relation for extensions. We call two extensions, $0\to V\to W_i\to Z\to 0$, $i=0,1$, {\it homotopic} if there is a path of Hilbert $C^*$-module morphisms $\Delta_t:Z\to Q(V)$, $t\in[0,1]$, such that $\Delta_i=\Delta_{W_i}$, $i=0,1$, and the map $t\mapsto\Delta_t(z)$ is norm-continuous for any $z\in Z$. For full Hilbert $C^*$-modules, unitary equivalence obviously implies homotopy equivalence.

\section{Classification problem}

Classification problem for Hilbert $C^*$-module extensions can be posed as follows. Let an essential $C^*$-algebra extension (\ref{1}) be given, and let $V$ and $Z$ be full Hilbert $C^*$-modules over $A$ and $B/A$ respectively. Classify all $W$ of the form (\ref{2}) up to unitary equivalence or up to homotopy. 

To get some insight, let us consider a simple example. 

\begin{example}
Let $Z=B/A$. For a Hilbert $C^*$-module $M$, denote by $\operatorname{Iso}(M)$ the set of all isometries on $M$, i.e. all adjointable operators $T$ on $M$ such that $T^*T=\id_M$. Since in this case any Busby invariant $\Delta:Z\to Q(V)$ is completely determined by its value $\Delta(1)$ on the unit of $Z=B/A$. Note that $\langle\Delta(1),\Delta(1)\rangle=\Delta(1)^*\Delta(1)=1$, and $\Delta(1)\in \operatorname{Iso}(Q(V))$ can be arbitrary, hence the homotopy classes of Hilbert $C^*$-module extensions are identified with the path-connected components of $\operatorname{Iso}(Q(V))$.  

\end{example}

\begin{example}
Let $A=K$ be the $C^*$-algebra of compact operators on a separable Hilbert space $H$, let $B=B(H)$ denote all bounded operators on $H$, and let $B/A=Q$ be the Calkin algebra. Take $V=K$, $Z=Q$. 

\end{example}

\begin{prop}
Homotopy classes of Hilbert $C^*$-module extensions $0\to K\to W\to Q\to 0$ are in one-to-one correspondence with the set $\mathbb Z\cup\infty$.

\end{prop}
\begin{proof}
Consider $1\in Q$ as a generating element of $Q$ as a module over itself. By ??? of \cite{BG}, $M(V)=B$, $Q(V)=Q$. The Busby invariant $\Delta$ is completely determined by an isometry $\Delta(1)=\sigma(F)\in Q$, where $F\in B(H)$ satisfies $F^*F=1$ modulo compacts and where $\sigma:B(H)\to Q$ is the canonical surjection. Perturbing $F$ by a compact operator, we may assume that $1-F^*F$ is a finitedimensional projection in $H$. If the projection $1-FF^*$ is finitedimensional too then the homotopy class of $F$ is determined by its index. If $1-FF^*$ is infinitedimensional then, up to compact perturbation, one may assume that $F^*F=1$ (still with infinitedemensional $1-FF^*$, and $FF^*$ is infinitedimensional as well). Any two such $F$'s are homotopic.  

\end{proof} 

\section{Commutative case}

Let an essential extension (\ref{1}) consist of commutative $C^*$-algebras. This means that $B=C(X)$ for some compact Hausdorff space $X$ with a dense open subset $U\subset X$ such that $C_0(U)=A$. Then $C(\partial U)=B/A$.

Let $\beta U$ denote the Stone-\v Cech compactification of $U$. Then
\begin{equation}\label{3-}
\begin{xymatrix}{
0\ar[r]& C_0(U)\ar[r]& C(\beta U)\ar[r]& C(\beta U\setminus U)\ar[r]& 0
}\end{xymatrix}
\end{equation}
is the universal essential $C^*$-algebra extension for 
\begin{equation}\label{3}
\begin{xymatrix}{
0\ar[r]& C_0(U)\ar[r]& C(X)\ar[r]& C(\partial U)\ar[r]& 0
}\end{xymatrix}
\end{equation}
with the canonical (injective) maps $\gamma:C(X)\to C(\beta U)$ and $\delta:C(\partial U)\to C(\beta U\setminus U)$ (the latter being the Busby invariant for (\ref{3-})).

Let $\eta$, $\xi$ be finitedimensional locally trivial vector bundles over $U$ and $\partial U$, respectively.
A Hermitian structure on $\xi$ endows the set $\Gamma(\xi)$ of all continuous sections of $\xi$ with a structure of a Hilbert $C^*$-module over $C(\partial U)$. Similarly, the set $\Gamma_0(\eta)$ of all continuous sections of $\eta$ vanishing at $\partial U$ has a structure of a Hilbert $C^*$-module over $C_0(U)$, given a Hermitian structure on $\eta$.

Our aim is to classify Hilbert $C^*$-module extensions of the form
\begin{equation}\label{4}
\begin{xymatrix}{
0\ar[r]& \Gamma_0(\eta)\ar[r]& W\ar[r]& \Gamma(\xi)\ar[r]& 0,
}\end{xymatrix}
\end{equation}
where $\eta$ is additionally assumed to be a subbundle of a trivial finitedimensional vector bundle (this is not automatic as $U$ is not compact). Let the fiber of this trivial vestor bundle be $\mathbb C^m$.

First, let us evaluate $Q(V)$ for $V=\Gamma_0(\eta)$. Let $P:U\to \Proj(\mathbb C^m)$ be a continuous mapping on $U$ taking values in the space of orthogonal projections on $\mathbb C^m$, such that $\Im P=\eta$ (equivalently, $P\cdot (C_0(U)^m)=\Gamma_0(\eta)$). As the matrix entries of $P$ are bounded functions on $U$, so they extend to continuous functions on $\beta U$. Being an orthogonal projection is an algebraic relation, hence the extension $\bar P$ of $P$ to $\beta U$ is a projection-valued mapping too. In particular, its restriction $\dot P$ onto $\beta U\setminus U$ is also a projection-valued mapping. It follows from the definition of $M(V)$ that if $V=A$ then $M(V)=M(A)$ (\cite{BG}). The same argument shows that $M(A^m)=M(A)^m$ and that $M(\bar P\cdot A^m)=\bar P\cdot M(A)^m$ for any projection $\bar P\in M_m(M(A))$, hence $M(P\cdot(C_0(U)^m))=\bar P\cdot C(\beta U)^m$. Thus $\dot P$ determines a vector bundle over $\beta U\setminus U$. Denote this vector bundle by $\zeta=\zeta(\eta)$.

Thus classification problem reduces to classification of Hilbert $C^*$-module morphisms $\Delta:\Gamma(\xi)\to\Gamma(\zeta)$. To this end we need the following proposition. It should be known to specialists, but we were unable to find a reference.

\begin{prop}
Let $\xi$, $\zeta$ be Hermitian vector bundles over compact Hausdorff spaces $Y$ and $Z$, respectively, and let $f:Z\to Y$ be a continuous mapping. The Hilbert $C^*$-module morphisms $\Gamma(\xi)\to\Gamma(\zeta)$ are in one-to-one correspondence with the set of fiberwise isometries from $f^*(\xi)$ to $\zeta$. 

\end{prop}
\begin{proof}
Let $\Delta:\Gamma(\xi)\to\Gamma(\zeta)$ be a Hilbert $C^*$-module morphism. 
Denote by $\xi_y$ the fiber of $\xi$ over the point $y\in Y$. For any $z\in Z$, there is a canonical isomorphism  $i_z:\xi_{f(z)}\to f^*(\xi)_z$. Let $v\in \xi_{f(z)}$, $v'=i_z(v)\in f^*(\xi)_z$. Given a point $z\in Z$, take a section $s_v:Y\to\xi$ such that $s_v(f(z))=v$. Set $D_\Delta(v')=\Delta(s_v)|_z\in\zeta_z$ (evaluation of the section $\Delta(s_v)$ at $z$). 

Let $\widehat{f}:C(Y)\to C(Z)$ be the $*$-homomorphism Gelfand-dual to $f:Z\to Y$, $\widehat{f}(\alpha)=\alpha\circ f$ for any $\alpha\in C(Y)$.

If $s'\in\Gamma(\xi)$ is another section such that $s'(f(z))=s_v(f(z))=v$ then $s_v-s'$ lies in the Hilbert $C^*$-submodule $\Gamma_y(\xi)$ consisting of all sections of $\xi$ vanishing at the point $y=f(z)$. As $\Delta$ is a Hilbert $C^*$-module morphism, so the function 
$$
\langle\Delta(s_v-s'),\Delta(s_v-s')\rangle=\widehat{f}(\langle s_v-s',s_v-s'\rangle)
$$ 
vanishes at the point $z$, hence $\Delta(s_v)|_z=\Delta(s')|_z$, therefore, the map $D_\Delta$ is well defined.

Note that a point $z\in Z$ defines not only the $C^*$-algebra extension $0\to C_0(Z\setminus z)\to C(Z)\to \mathbb C\to 0$, but also the corresponding Hilbert $C^*$-module extension $0\to \Gamma_z(\zeta)\to\Gamma(\zeta)\to\zeta_z\to 0$. In particular, the norm in the fiber $\zeta_z$ coinsides with the quotient Hilbert $C^*$-module norm in $\Gamma(\zeta)/\Gamma_z(\zeta)$. If $u\in\zeta_{z}$ then 
$$
\|u\|^2=\inf_{s\in\Gamma(\zeta):s(z)=u}\|s\|^2=\inf_{s\in\Gamma(\zeta):s(z)=u}\|\langle 
s,s\rangle\|=|\langle s,s\rangle|_{z}|
$$ 
for any $s\in\Gamma(\zeta)$ with $s(z)=u$.

As
$$
\|D_\Delta(v')\|^2=|\langle D_\Delta(v'),D_\Delta(v')\rangle|_z|=|\langle \Delta(s_v),\Delta(s_v)\rangle|_z|=
|\langle s_v,s_v\rangle|_{f(z)}|=\|v\|^2=\|v'\|^2,
$$
so $D_\Delta$ is isometric.

Conversely, let $D:f^*(\xi)\to\zeta$ is a fiberwise isometry of vector bundles over $Z$,
$$
\begin{xymatrix}{
f^*(\xi)\ar[rr]^-{D}\ar[dr]&&\zeta\ar[dl]\\
&Z.&
}\end{xymatrix}
$$

For a section $s\in\Gamma(\xi)$, $s:Y\to\xi$, define $\Delta_D$ by $\Delta_D(s)=D\circ s\circ f:Z\to\zeta$. This gives a map $\Delta_D:\Gamma(\xi)\to\Gamma(\zeta)$. Since $D$ is isometric,
$$
\langle\Delta_D(s),\Delta_D(s')\rangle=\langle D\circ s\circ f,D\circ s'\circ f\rangle=\langle s\circ f,s'\circ f\rangle=\widehat{f}(\langle s,s'\rangle)
$$
for any sections $s,s'\in\Gamma(\xi)$, therefore, $\Delta_D$ is a Hilbert $C^*$-module morphism.

Finally, it is easy to see that $\Delta_{D_\Delta}=\Delta$ and $D_{\Delta_D}=D$.

\end{proof}

\begin{cor}
The set of homotopy classes of Hilbert $C^*$-module extensions of the form (\ref{4}), where $\xi$ and $\eta$ are vector bundles that are direct summands in trivial finitedimensional bundles, are in one-to-one correspondence with the set of homotopy classes of fiberwise isometries $f^*(\xi)\to\zeta$, where $\zeta$ was consturcted out of $\eta$ as above, over the base $\beta U\setminus U$.

\end{cor}

As a corollary, we obtain the following classification. Let $V_{k,m}$ denote the Stieffel manifold of $k$-frames in $\mathbb C^m$.

\begin{thm}
Let $Z$ be a free module of rank $k$ over $C(\partial U)$, $V$ a Hilbert $C^*$-module over $C_0(U)$ such that $Q(V)$ is a free module of rank $m$ over $C(\beta U\setminus U)$ (this holds, e.g., when $V=C_0(U)^m$). Then the set of homotopy classes of Hilbert $C^*$-module extensions with given $V$ and $Z$ and with the underlying $C^*$-algebra extension (\ref{3}) is in one-to-one correspondence with the set $[\beta U\setminus U, V_{k,m}]$ of homotopy classes of maps from the Stone-\v Cech corona $\beta U\setminus U$ to $V_{k,m}$.   

\end{thm}
\begin{proof}
Fix an orthonormal basis in the trivial vector bundle over $\partial U$. Then the Busby invariant $\Delta$ is completely determined by the image of this basis in the trivial vector bundle over $\beta U\setminus U$.

\end{proof}

In the special case of one-dimensional trivial vector bundles we have the following classification.

\begin{cor}
Let $V=C_0(U)$ and $Z=C(\partial U)$. Then the homotopy classes of Hilbert $C^*$-module extensions of the form $0\to C_0(U)\to W\to C(\partial U)\to 0$ with the underlying $C^*$-algebra extension (\ref{3}) are in one-to-one correspondence with the set $[\beta U\setminus U,\mathbb S^1]$ of homotopy classes of maps from the Stone-\v Cech corona to the circle.

\end{cor}

\begin{remark}
The set $[\beta U\setminus U,\mathbb S^1]$ is canonically isomorphic to the first \v Cech cohomology group $H^1(\beta U\setminus U)$.

\end{remark}

\begin{example}
Let 
$$
U=\{z\in\mathbb C:|z|<1\}, \quad \overline{U}=\{z\in\mathbb C:|z|\leq 1\},
$$ 
and let $X\cong\mathbb S^2$ be the one-point compactification of $U$. This gives rise to a $C^*$-algebra extension
\begin{equation}\label{5}
\begin{xymatrix}{
0\ar[r]& C_0(U)\ar[r]& C(\mathbb S^2)\ar[r]& \mathbb C\ar[r]& 0.
}\end{xymatrix}
\end{equation}

Let $V=C_0(U)$, $Z=\mathbb C$ be free rank one Hilbert $C^*$-modules over $C_0(U)$ and $\mathbb C$, respectively. 

Let $\omega:\partial U\to\partial U=\mathbb S^1$ be a continuous function with winding number $k$. Set 
$$
W_k=\{\alpha\in C(\overline{U}):\alpha|_{\partial U}=\lambda\omega\ \mbox{for some\ }\lambda\in\mathbb C\}.
$$ 
with the obvious module action of $C(\mathbb S^2)$. It is clear that $\langle W_k,W_k\rangle= C(\mathbb S^2)$ considered as an ideal in $C(\overline{U})$. Thus $W_k$ is a full Hilbert $C^*$-module over $C(\mathbb S^2)$. Set $\Pi:W_k\to\mathbb C$ by $\Pi(\alpha)=\lambda$ (it is simple to check that this map is well defined). Then we get a Hilbert $C^*$-module extension
\begin{equation}\label{6}
\begin{xymatrix}{
0\ar[r]& C_0(U)\ar[r]^-{\Phi}& W_k\ar[r]^-{\Pi}& \mathbb C\ar[r]& 0
}\end{xymatrix}
\end{equation}
with the underlying $C^*$-algebra extension (\ref{5}), where $\Phi$ is the inclusion.

\end{example}

\begin{thm}
Hilbert $C^*$-module extensions (\ref{6}) with different $k$ are not homotopy equivalent to each other.

\end{thm}
\begin{proof}
First note that the Hilbert $C^*$-module morphisms $\Gamma$ and $\Delta$ factorize through $C(\overline{U})$ and $C(\mathbb S^1)$, respectively:
$$
\begin{xymatrix}{
0\ar[r]& C_0(U)\ar[r]\ar@{=}[d]& W_k\ar[r]^-{\Pi}\ar[d]^-{\Gamma'}& \mathbb C\ar[r]\ar[d]^-{\Delta'}& 0\\
0\ar[r]& C_0(U)\ar[r]\ar@{=}[d]& C(\overline{U})\ar[r]^-{\Pi'}\ar[d]^-{\Gamma''}& C(\mathbb S^1)\ar[r]\ar[d]^-{\Delta''}& 0\\
0\ar[r]& C_0(U)\ar[r]& C_b(U)\ar[r]^-{\Pi''}& C_b(U)/C_0(U)\ar[r]& 0,
}\end{xymatrix}
$$
where $C_b(U)$ is the $C^*$-algebra (and a module) of all bounded continuous functions on $U$, and the maps  $\Gamma''$ and $\Delta''$ are canonical inclusions and $\Pi'$ is the restriction map.

It follows from the definition of $W_k$ and from commutativity of the above diagram that $\Delta(1)=\Delta'(1)=\omega$. Let $\widehat{\Delta}'':\beta U\setminus U\to\mathbb S^1$ be the map Gelfand-dual to $\Delta''$. The function $\omega$ represents an element $[k]$ of $[\mathbb S^1,\mathbb S^1]$. Then $\Delta(1)=\widehat{\Delta}''\circ\omega$ represents an element of $[\beta U\setminus U,\mathbb S^1]$. It remains to show that the group homomorphism 
\begin{equation}\label{0}
\mathbb Z\cong [\mathbb S^1,\mathbb S^1]\to [\beta U\setminus U,\mathbb S^1]=H^1(\beta U\setminus U),
\end{equation}
induced by $\widehat{\Delta}''$ is injective.

To this end, set $U_r=\{z\in\mathbb C:|z|<r\}$, $0<r<1$. Then $\beta(U\setminus U_r)=\beta U\setminus U_r$ and $\beta U\setminus U=\cap_r\beta(U\setminus U_r)$. A standard direct limit argument shows that $[\beta U\setminus U,\mathbb S^1]=\di\lim_{r\to 1} [\beta(U\setminus U_r),\mathbb S^1]$.

The standard argument of C. H. Dowker \cite{D} shows that the map $[\beta (U\setminus U_r],\mathbb S^1]\to[U\setminus U_r,\mathbb S^1]\cong[\mathbb S^1,\mathbb S^1]\cong\mathbb Z$ induced by the Stone-\v Cech compactification $U\setminus U_r\subset\beta(U\setminus U_r)$ is surjective. It is also clear that this surjection is compatible with the map $[\beta(U\setminus U_r),\mathbb S^1]\to[\beta(U\setminus U_s),\mathbb S^1]$ induced by the inclusion $U\setminus U_s\subset U\setminus U_r$, $r<s<1$. Therefore, there is a surjective map $[\beta U\setminus U,\mathbb S^1]\to [\mathbb S^1,\mathbb S^1]$ and its composition with the map (\ref{0}) is the identity, hence (\ref{0}) is injective. 

\end{proof}

\vspace{2cm}

\end{document}